\documentclass[leqno,a4paper,11pt]{article}
\usepackage{amsmath,amssymb,ascmac,latexsym,amsthm}
\usepackage[dvips]{graphicx}
\theoremstyle{definition}
\newtheorem{Assu}{Assumption}

\newtheorem{Th}{Theorem}
\newtheorem{Def}{Definition}
\newtheorem{rem}{Remark}
\newtheorem{Lem}{Lemma}
\newtheorem{Prop}[Lem]{Proposition}

\title{Scattering theory for Schr\"odinger equations with time-dependent short-range potentials via wave packet transform}
\author{Taisuke Yoneyama and Keiichi Kato}
\date{\today}
\newcommand{\dint}{\displaystyle\int}

\begin{document}
\maketitle
%
%
\section{Introduction}

In this paper, we consider the existence and the asymptotic completeness of the wave operators for Schr\"odinger equations with time-dependent potentials which are short-range in space. 
For time-independent short-range potentials, the wave operators have been studied since 1950s by methods of using properties of spectrum (e.g., J. Cook \cite{cook}, S. Kuroda \cite{kuroda}).
Then the ranges of the wave operators are characterized by the absolutely continuous subspace.
V. Enss \cite{Enss} has proved the existence of the wave operators and that the non-trivial singular continuous subspace of $H$ is nothing.
Then the ranges of the wave operators are characterized by the continuous subspace of the operator.
For time-dependent short-range potentials,
D. R. Yafaev \cite{Yafaev} has shown the existence of the wave operators with time-dependent potentials,
and that the ranges coincide with the entire space under the suitable conditions.
H. Kitada--K. Yajima \cite{KY} shows the existence the wave operators and the modified wave operators with time-dependent potentials which are short-range in space
and has characterized the ranges of the (modified) wave operators.
In this paper, we prove the existence and characterize the ranges of the wave operators
for Schr\"odinger equations with time-dependent potentials which are short-range in space
by introducing the wave packet transform.

We consider the following problem:
\begin{align}
\label{siki-1-nonini}
i\frac{\partial}{\partial t}u = H(t)u, \quad H(t)= H_0+V(t),\quad H_0
=-\frac12\displaystyle\sum_{j=1}^{n}\displaystyle\frac{\partial^2}{\partial x_j^2}=-\frac12\Delta
\end{align}
in the Hilbert space $\mathcal{H}=L^2(\mathbb{R}^n)$, and the domain $D(H_0)=H^2(\mathbb{R}^{n})$ where $H^2(\mathbb{R}^{n})$ is Sobolev space of order two.

We assume that $V$ is satisfied with the following conditions called ``short-range''.\\
\begin{Assu}
\hspace{11pt}(i) $V(t,x)$ is a real-valued function of $(t,x)\in\mathbb{R}\times\mathbb{R}^{n}$.\\
\hspace{9pt}(ii)There exists a constant $\delta>1$ such that there exists $C>0$ satisfying
\begin{align*}
|V(t,x)|\leq 
C(1+|x|)^{-\delta},\quad(t,x)\in\mathbb{R}\times\mathbb{R}^{n}.
\end{align*}
\end{Assu}
Under Assumption (A), $H(t)$ is self adjoint in $L^2(\mathbb{R}^{n})$ with the domain $D(H(t))=H^2(\mathbb{R}^{n})$ for each $t\in\mathbb{R}$.
Then there exist unitary evolution operators $U(t,t_0)$ satisfying the following conditions (c.f., T. Kato \cite{TKato}).

\noindent\hspace{11pt}(i) $U(t,t_0)f$ is continuous function of $t$ and
\[
U(t,\tau')U(\tau',\tau)=U(t,\tau),\,U(t,t)=Id\quad t,\tau',\tau\in\mathbb{R}.
\]
\hspace{9pt}(ii) $U(t,\tau)f$ is continuously differentiable and
\[
\frac{\partial}{\partial t}U(t,\tau)f=-iH(t)U(t,\tau)f\quad f\in L^2(\mathbb{R}^{n}).
\]

Let $a\geq0$ and $R>0$.
We define $\Gamma=\Gamma_{a,R}$ by
\[
\Gamma_{a,R}=\left\{(x,\xi)\in\mathbb{R}^n\times\mathbb{R}^n\Big\vert|\xi|\leq a\mbox{ or }|x|\geq R\right\},
\]
and $\mathcal{S}_{scat}$ by
\[
\mathcal{S}_{scat}=\left\{\Phi\in\mathcal{S}\Big\vert\|\Phi\|_{L^2(\mathbb{R}^{n})}=1 \mbox{ and }\hat{\Phi}(0)\neq0\right\}.
\]
\begin{Def}\label{wpt-def}
(Wave packet transform)

Let $\varphi\in\mathcal{S}\setminus\left\{ 0\right\}$ and $f\in\mathcal{S'}$.
We define the wave packet transform $W_\varphi f(x,\xi)$ of $f$ 
with the wave packet generated by a function $\varphi$ as follows:
\[
W_\varphi f(x,\xi)=\int_{\mathbb{R}^{n}}\overline{\varphi(y-x)}f(y)e^{-iy\xi}dy,
\quad (x,\xi)\in\mathbb{R}^{n}\times\mathbb{R}^{n}.
\]

We also define the inverse of the wave packet transform $W_{\varphi}^{-1}$ 
for a function $F(x,\xi)$ on $\mathbb{R}^{n}\times\mathbb{R}^{n}$ as follows:
\[
W_\varphi^{-1} F(x)=\frac{1}{(2\pi)^n\|\varphi\|_{L^2(\mathbb{R}^{n})}^2}\int\int_{\mathbb{R}^{2n}}\varphi(x-y)F(y,\xi)e^{ix\xi}dyd\xi,
\quad x\in\mathbb{R}^{n}.
\]
\end{Def}
\begin{Def}\label{scat-def}
(Scattering subspace)

Let $\Phi\in\mathcal{S}_{scat}$.
We define $\tilde{D}_{scat}^{\pm,\Phi}(\tau)$ by the set of all functions in $L^2$ satisfying
that there exist $a>0$ and $R>0$ such that
\[
\chi_{\Gamma_{a,R}}W_{\Phi(t)}[U(t,\tau)f](x+(t-\tau)\xi,\xi)\rightarrow 0 \,(t\rightarrow\pm\infty)\quad\mbox{in } L^2(\mathbb{R}^{n}\times\mathbb{R}^{n})
\]
where $\Phi(t)=e^{-i(t-\tau)H_0}\Phi$.

The scattering subspaces $D_{scat}^{\pm,\Phi}(\tau)$ is defined by the closure of $\tilde{D}_{scat}^{\pm,\Phi}(\tau)$.
\end{Def}

The aim of this paper is to prove the following theorem.
\begin{Th}\label{main-theorem}
Suppose that (A) be satisfied.
Then the wave operators
\begin{align*}
W_{\pm}(\tau)=\mbox{s-}\hspace{-2mm}\lim_{t\rightarrow\pm\infty}U(\tau,t)e^{-i(t-\tau)H_0}
\end{align*}
exist for any $\tau\in\mathbb{R}$ 
and their ranges $\mathcal{R}(W_{\pm}(\tau))$ coincide with $D_{scat}^{\pm,\Phi}(\tau)$ for any $\Phi\in\mathcal{S}_{scat}$.

In particular, $D_{scat}^{\pm,\Phi}(\tau)$ is independent of $\Phi$.
So in the sequel, we write $D_{scat}^{\pm,\Phi}(\tau)$ as $D_{scat}^{\pm}(\tau)$.
\end{Th}


We use the following notations throughout the paper.
$i=\sqrt{-1},\,n\in\mathbb{N}$. We write $\partial_{x_j}=\partial/\partial_{x_j},\, \partial_t=\frac{\partial}{\partial t},
\,L^2=L^2(\mathbb{R}^{n}),\,\|\cdot\|=\|\cdot\|_{L^2},\,(\cdot,\cdot)=(\cdot,\cdot)_{L^2}$,
$\langle t\rangle=1+|t|$, $\|f\|_{H^l}=\sum_{|\alpha+\beta|=l}\|x^\beta\partial^\alpha_x f\|.$
$\mathcal{F}$ is the Fourier transform:
\[
\mathcal{F}f(\xi)=\hat{f}(\xi)=\int_{\mathbb{R}^{n}} e^{-ix\cdot\xi}f(x)dx.
\]
We also write $\mathcal{F}^{-1}f(x)=(2\pi)^{-n}\int e^{ix\cdot\xi}f(\xi)d\xi$,
and $\chi_A(x)$ denotes the characteristic function of a measurable set $A$.
For an operator $T$, $D(T)$ and $\mathcal{R}(T)$ denote the domain and range of $T$.
$\mathcal{H}_{p}(T)$ denotes pure point subspace of a self-adjoint operator $T$.

We can treat both $x$ and $\xi$ in the phase space at the same time by introducing the wave packet transform.
In order to prove asymptotic completeness of the wave operators,
the phase space decomposition into near the classic orbits of the particles or not
requires complicated arguments.
Specifically, for the phase space decomposition operators corresponding to $P_+$, $P_-$, we have to prove $P_-U(t,0)\rightarrow0$
and $P_+U(t,0)\sim P_+e^{-itH_0}$.
Enss method and other studies (\cite{Enss}, \cite{kuroda2}, \cite{perry}, \cite{KY}) define the phase space decomposition operators as a pseudo-differential operator.
But our method by the wave packet transform can treat both $x$ and $\xi$ in the phase space at the same time,
we only have to the multiplication operators of the characteristic functions.
Thus we can decompose the phase space by only using the density argument.


\begin{rem}\label{rem-asu}
We see that the condition $V\in L^\infty(\mathbb{R}^{n+1})$ and that
there exists $h\in L^1([0,\infty);dR)$ such that
\[
\|\chi_{|x|>R}V(t,x)\|_{L^\infty(\mathbb{R}^{n+1})}\leq h(R)
\]
can replace Assumption (A), (iii).
\end{rem}

\begin{rem}\label{rem-eigen}
If $V(t,x)=V(x)$(, $H(t)=H$) is independent of $t$, we see that $U(t,\tau)=e^{-i(t-\tau)H}$.
Then we obtain $D_{scat}^{\pm}(\tau)\perp\mathcal{H}_{p}(H)$ for all $\tau\in\mathbb{R}$.
\end{rem}

\begin{rem}\label{rem-ana}
In the case $n\geq2$, we obtain an another characterization of $\mathcal{R}(W_\pm(\tau))$.

Let $a\geq0,\,0<\sigma\leq1$ and $\Phi\in\mathcal{S}_{scat}$.
We define $\tilde{\Gamma}_{a,\sigma}$ by
\[
\tilde{\Gamma}_{a,\sigma}=\left\{(x,\xi)\in\mathbb{R}^n\times\mathbb{R}^n\Big\vert|\xi|\leq a\mbox{ or }|\cos\theta(x,\xi)|\geq\sigma\right\}
\]
where $\cos\theta(x,\xi)=(x\cdot\xi)/|x||\xi|$,
and $A_{scat}^{\pm,\Phi}(\tau)$ is the closure of the set of all functions in $L^2$ satisfying
that there exist $a>0$ and $\sigma\in(0,1)$ such that
\[
\chi_{\tilde{\Gamma}_{a,\sigma}}W_{\Phi(t)}[U(t,\tau)f](x+(t-\tau)\xi,\xi)\rightarrow 0 \,(t\rightarrow\pm\infty)\quad\mbox{in } L^2(\mathbb{R}^{n}\times\mathbb{R}^{n}).
\]

Then we get for any $\Phi\in\mathcal{S}_{scat}$
\begin{align}
\label{AeqD}
A_{scat}^{\pm,\Phi}(\tau)=\mathcal{R}(W_\pm(\tau))=D_{scat}^{\pm}(\tau).
\end{align}
\end{rem}
The plan of this paper is as follows.
In section 2, we mention the properties of the wave packet transform.
In section 3, we give the proof of the existence of the wave operators.
In section 4, we give the proof of the characterization of the ranges of the wave operators.
In section 5, we prove the remarks.
%
%
\section{Wave packet transform}
In this section, we introduce the properties of the wave packet transform.

\begin{Prop}\label{wpt-prop}
Let $\varphi\in\mathcal{S}\setminus\left\{ 0\right\}$ and $f\in\mathcal{S'}$.

Then the wave packet transform $W_\varphi f(x,\xi)$ has 
the following properties:\\
(i) $W_\varphi f(x,\xi)\in C^\infty(\mathbb{R}_x^{n}\times\mathbb{R}_\xi^{n})$.\\
(ii) If $u,\varphi$ depend on $t$, we have
\begin{align*}
W_{\varphi(t)}[\partial_xf](t,x,\xi)
&=-i\xi W_{\varphi(t)}[f](t,x,\xi)+W_{\partial_x\varphi(t)}[f](t,x,\xi), \mbox{ in }\mathcal{S'},\\
W_{\varphi(t)}[\partial_tf](t,x,\xi)
&=\partial_tW_{\varphi(t)}[f](t,x,\xi)-W_{\partial_t\varphi(t)}[f](t,x,\xi), \mbox{ in }\mathcal{S'}.
\end{align*}
(iii) If $f,g\in L^2$, $\psi\in\mathcal{S}\setminus\left\{ 0\right\}$,
\begin{align*}
(W_\varphi f,W_\psi g)_{L^2(\mathbb{R}_x^{n}\times\mathbb{R}_\xi^{n})}
&=\overline{(\varphi,\psi )}( f,g )=(\psi,\varphi)( f,g ).
\end{align*}
(iv) The inversion formula
$
W_\varphi^{-1}[W_\varphi f]=f
$
holds.
\end{Prop}
\begin{proof}
See \cite{KKS2}.
\end{proof}
In this section, let $\varphi_0\in\mathcal{S}\setminus\left\{ 0\right\}$ and $u_0\in L^2$.
The idea of considering the wave packet transform with a time-dependent wave packet is quoted from \cite{KMS}.

We transform $(\ref{siki-1-nonini})$ with initial data:
\begin{align}
\label{siki-1}
&\begin{cases}
i\partial_tu+\frac{1}{2}\Delta u -V(t,x)u = 0, \quad (t,x)\in\mathbb{R}\times\mathbb{R}^{n},\\
u(t_0)=u_0,\hspace{90pt} x\in\mathbb{R}^n
\end{cases}\end{align}
via the wave packet transform with a wave packet $\varphi(t)=e^{-itH_0}\varphi_0$.
From Proposition \ref{wpt-prop}, we have
\begin{align*}
W_{\varphi(t)}[\Delta u](t,x,\xi)
&=\dint\overline{\varphi(t,y-x)}\Delta u(y)e^{-iy\xi}dy\\
&=\dint\Delta \overline{\varphi(t,y-x)}u(y)e^{-iy\xi}dy\\
&+\dint(-2i\xi\cdot\nabla_y)\overline{\varphi(t,y-x)}u(y)e^{-iy\xi}dy
-|\xi|^2\dint\overline{\varphi(t,y-x)}u(y)e^{-iy\xi}dy\\
&=W_{\Delta\varphi(t)}u(t,x,\xi)+2i\xi\cdot\nabla_xW_{\varphi(t)}u(t,x,\xi)
-|\xi|^2W_{\varphi(t)}u(t,x,\xi).
\end{align*}
Since $W_{\varphi(t)}[i\partial_tu](t,x,\xi)=i\partial_tW_{\varphi(t)}u(t,x,\xi)
+W_{i\partial_t\varphi(t)}u(t,x,\xi)$, $(\ref{siki-1})$ is transformed to
\begin{align}
\label{siki-2}
&\begin{cases}
\Big(i\partial_t+i\xi\cdot\nabla_x-\frac{1}{2}|\xi|^2\Big)W_{\varphi(t)}u(t,x,\xi) = R_{\varphi_0}(t,x,\xi;u_0),\\
W_{\varphi(t_0)}u(t_0,x,\xi)=W_{\varphi_0}u_0(x,\xi),
\end{cases}
\end{align}
where
\[
R_{\varphi_0}(t,x,\xi;u_0)=\int\overline{e^{-isH_0}\varphi_0(y-x)}V(t,y)U(t,0)u_0(y)e^{-iy\xi}dy.
\]

Regarding $R_{\varphi_0}(t,x,\xi;u_0)$ of $(\ref{siki-2})$ as an inhomogeneous term, we have
\begin{align*}
W_{\varphi(t)}u(t,x,\xi)=&e^{-i\frac{1}{2}(t-t_0) |\xi|^2}
W_{\varphi_0}u_0(x-(t-t_0)\xi,\xi)\\
&\hspace{10pt}-i\int_{t_0}^t e^{-i\frac{1}{2}(t-s) |\xi|^2}
R_{\varphi_0}(s,x-(t-s)\xi,\xi;u_0)ds.
\end{align*}
This shows that the wave packet transform enables us to present the propagator with the integral equation.
Therefore for $t,t'>0$ and $\psi\in L^2$, we obtain the following calculations:
\begin{align}
\label{siki-psi}
W_{\varphi(t)}&[e^{-itH_0}\psi](x,\xi)=e^{-i\frac{1}{2}t|\xi|^2}W_{\varphi_0}\psi(x-t\xi,\xi),
\end{align}
\begin{align}
\label{siki-psi2-e}
W_{\varphi_0}&[e^{itH_0}U(t,0)\psi](x,\xi)\nonumber\\
&=e^{i\frac{1}{2}t|\xi|^2}W_{\varphi(t)}[U(t,0)\psi](x+t\xi,\xi)\nonumber\\
&=e^{i\frac{1}{2}t|\xi|^2}\cdot \Big(e^{-\frac{1}{2}it|\xi|^2}W_{\varphi_0}\psi((x+t\xi)-t\xi,\xi)\\
&\hspace{40mm}-i\int_0^t e^{-i\frac{1}{2}(t-s) |\xi|^2}
R_{\varphi_0}(s,((x+t\xi)-(t-s)\xi;\psi)ds\Big)\nonumber\\
&=W_{\varphi_0}\psi(x,\xi)-i\int_0^t e^{i\frac{1}{2}s |\xi|^2}R_{\varphi_0}(s,x+s\xi,\xi;\psi)ds,\nonumber
\end{align}
\begin{align}
\label{siki-psi2}
W_{\varphi(t)}&[U(t,t')e^{-it'H_0}\psi](x,\xi)\nonumber\\
&=e^{-\frac{1}{2}i(t-t')|\xi|^2}W_{\varphi(t')}[e^{-it'H_0}\psi](x-(t-t')\xi,\xi)\\
&\hspace{40mm}+i\int_t^{t'} e^{-i\frac{1}{2}(t-s) |\xi|^2}
R_{\varphi_0}(s,x+(t-s)\xi,\xi;e^{-it'H_0}\psi)ds\nonumber\\
&=W_{\varphi_0}\psi(x-t\xi,\xi)+i\int_t^{t'} e^{-i\frac{1}{2}(t-s) |\xi|^2}
R_{\varphi_0}(s,x+(t-s)\xi,\xi;e^{-it'H_0}\psi)ds.\nonumber
\end{align}

%
%
\section{Existence of the wave operators}
In this section, we prove the existence of the wave operators by the wave packet transform in the previous section.

Let $0<\rho<1$.
We define $V_\rho$ as follows:
\begin{align*}
V_\rho(t,x)=\chi_0(\frac{1}{\rho \langle t\rangle}x)V(t,x),
\end{align*}
where $\chi_0\in C^\infty(\mathbb{R}^n)$ satisfies $\chi_0(x)=1\,(|x|\geq1),\,\chi_0(x)=0\,(|x|\leq\frac{1}{2})$.
Then there exists $C>0$ such that
\begin{align}
\label{v-2}
|V_\rho(t,x)|\leq C\langle t\rangle^{-\delta},
\end{align}
for any $t\in\mathbb{R}$, $x\in\mathbb{R}^n$. 
When $|x|\geq\rho \langle t\rangle$,
\begin{align}
\label{v-1}
V_\rho(t,x)=V(t,x).
\end{align}

The following well-known lemma is used in the proof of Proposition \ref{no2-term}.
\begin{Lem}\label{kuroda-lem}
Let $f\in\mathcal{S}$.
If supp$\hat{f}\subset K$ with some compact set $K$ which does not contain the origin, for an open set $K'\supset 2K$ and $l,k\geq0$ there exists $C=C_{K,K',l,k}$ such that
\begin{align*}
|e^{-itH_0}f(x)|\leq C\langle t\rangle^{-k}\langle x\rangle^{-l}\|f\|_{H^{k+l}}\quad(x/t\notin K',t\neq 0).
\end{align*}
\end{Lem}
\begin{proof}
See \cite{kuroda2}.
\end{proof}
Using the above lemma, we obtain the following lemma.
\begin{Prop}\label{no2-term}
Let $a>0$ and $R>0$. There exist $r_0>0$ and $T>0$ such that for all $r\in (0,r_0]$
\begin{align*}
\big\|R_{\varphi_0}(s,x+s\xi,\xi;\psi)\big\|_{L^2(\mathbb{R}^{2n}\setminus\Gamma_{a,R})}\leq C\langle s\rangle^{-\delta}\|\psi\|
\end{align*}
for some $C=C_{a,R,r}>0$ and for any $s\geq T$ where $\varphi_0\in\mathcal{S}\setminus\{0\}$ satisfying supp$\hat{\varphi_0}\subset\{r/2<|\xi|<r\}$.
\end{Prop}
\begin{proof}
For $c\in(0,a)$, we see that
\begin{align*}
|x\pm s\xi|
\geq as-R
\geq cs
\geq \frac12c\langle s\rangle
\end{align*}
for any $(x,\xi)\in\mathbb{R}^{2n}\setminus\Gamma_{a,R}$ and $s\geq T=\max\{R/(a-c),1\}$.

Taking $\rho=c/6$, we put
\begin{align*}
R_{\varphi_0}(s,x,\xi;\psi)
&=\int\overline{e^{-isH_0}\varphi_0(y-x)}V(s,y)U(s,0)\psi(y)e^{-i\xi y}dy\\
&\leq\int_{\{|y-x|\leq 2\rho s\}}\overline{e^{-isH_0}\varphi_0(y-x)}V(s,y)U(s,0)\psi(y)e^{-i\xi y}dy\nonumber\\
&+\int_{\{|y-x|> 2\rho s\}}\overline{e^{-isH_0}\varphi_0(y-x)}V(s,y)U(s,0)\psi(y)e^{-i\xi y}dy\nonumber\\
&= I_1(s,x,\xi;\psi)+I_2(s,x,\xi;\psi).
\end{align*}

By ($\ref{v-2}$), ($\ref{v-1}$), change of variables as $x'=x+s\xi$ and Plancherel theorem, we have
\begin{align*}
\big\|I_1&(s,x+s\xi,\xi;\psi)\big\|_{L^2(\mathbb{R}^{2n}\setminus\Gamma_{a,R})}\\
&=\|\int_{\{|y-(x+s\xi)|\leq 2\rho s\}}\overline{e^{-isH_0}\varphi_0(y-(x+s\xi))}V_\rho(s,y)U(s,0)\psi(y)e^{-i\xi y}dy\big\|_{L^2(\mathbb{R}^{2n}\setminus\Gamma_{a,R})}\\
&\leq\|\int_{\{|y-x'|\leq 2\rho s\}}\overline{e^{-isH_0}\varphi_0(y-x')}V_\rho(s,y)U(s,0)\psi(y)e^{-i\xi y}dy\big\|_{L^2(\mathbb{R}_{x'}^n\times\mathbb{R}_\xi^n)}\\
&=\|\chi_{\{|y-x|\leq 2\rho s\}}\overline{e^{-isH_0}\varphi_0(y-x)}V_\rho(s,y)U(s,0)\psi(y)\big\|_{L^2(\mathbb{R}_x^n\times\mathbb{R}_y^n)}\\
&=\|\overline{e^{-isH_0}\varphi_0(y-x)}V_\rho(s,y)U(s,0)\psi(y)\big\|_{L^2(\mathbb{R}_x^n\times\mathbb{R}_y^n)}\\
&\leq C\langle s\rangle^{-\delta}\|\psi\|
\end{align*}
for $s\geq T$.
By Lemma $\ref{kuroda-lem}$ and taking $r_0=\rho$, we have
\begin{align*}
\big\|I_2&(s,x+s\xi,\xi;\psi)\big\|_{L^2(\mathbb{R}^{2n}\setminus\Gamma_{a,R})}\\
&\leq\big\|\int_{\{|y-x'|> 2\rho s\}}\overline{e^{-isH_0}\varphi_0(y-x')}V(s,y)U(s,0)\psi(y)e^{-i\xi y}dy\big\|_{L^2(\mathbb{R}_{x'}^n\times\mathbb{R}_{\xi}^n)}\nonumber\\
&=\big\|\chi_{\{|y-x|> 2\rho s\}}\overline{e^{-isH_0}\varphi_0(y-x)}V(s,y)U(s,0)\psi(y)\big\|_{L^2(\mathbb{R}_x^n\times\mathbb{R}_y^n)}\nonumber\\
&\leq C\langle s\rangle^{-k}\|\hat{\varphi_{0}}\|_{H^{k+l}}\big\|\langle y-x\rangle^{-l}V(s,y)U(s,0)\psi(y)\big\|_{L^2(\mathbb{R}_x^n\times\mathbb{R}_y^n)}\nonumber\\
&= C\langle s\rangle^{-k}\|\hat{\varphi_{0}}\|_{H^{k+l}}\big\|\langle x\rangle^{-l}\big\|\big\|U(s,0)\psi(y)\big\|\nonumber\\
&\leq C\langle s\rangle^{-\delta}\|\psi\|\nonumber.
\end{align*}
\end{proof}

We give a proof of the existence of the wave operators $W_\pm(\tau)$.
The following lemma is well known and has been already proved, however, we give the proof using wave packet transform.

\begin{Prop}\label{PropExis}
Suppose that (A) be satisfied.
Then the wave operators
$W_{\pm}(\tau)$
exist for any $\tau\in\mathbb{R}$.
\end{Prop}

\begin{proof}
Substituting $V(t-\tau,x)$ for $V(t,x)$, it suffices to show the case $\tau=0$.
We prove the existence in the case $t\rightarrow+\infty$ only.
In the other cases, they can be proved similarly.

Let $\Phi\in\mathcal{S}_{scat}$ and $u_0\in L^2$.
It suffices to prove the existence for $W_\Phi u_0\in C_0^\infty(\mathbb{R}^{2n}\setminus\{|\xi|=0\})$.
Indeed, let $\varepsilon>0$ fixed.
Since $C_0^\infty(\mathbb{R}^{2n}\setminus\{|\xi|=0\})$ is dense in $L^2(\mathbb{R}^{2n})$, there exists $\omega\in C_0^\infty(\mathbb{R}^{2n}\setminus\{|\xi|=0\})$ satisfying
$
\|W_\Phi u_0-\omega\|_{L^2(\mathbb{R}^{2n})}<\varepsilon.
$

Since $\|\Phi\|=1$, putting $\tilde{u_0}=W_\Phi^{-1}\omega$ we have
\begin{align*}
W_\Phi\tilde{u_0}=\omega,\quad\|\tilde{u_0}-u_0\|&=\|W_\Phi\tilde{u_0}-W_\Phi u_0\|_{L^2(\mathbb{R}^{2n})}\\
&=\|\omega-W_\Phi u_0\|_{L^2(\mathbb{R}^{2n})}<\varepsilon.
\end{align*}

Then there exist $a>0$ and $R>0$ such that supp\hspace{1pt}$\omega\subset\mathbb{R}^{2n}\setminus\Gamma_{a,R}$.
Let  $\varphi_0\in\mathcal{S}\setminus\{0\}$ satisfying
\[
\mbox{supp}\hat{\varphi_0}\subset\left\{\frac{r}2<|\xi|<r\right\}\mbox{ with }0<r\leq r_0,\quad |( \Phi,\varphi_0)|>0
\]
where $r_0$ is the value satisfying Proposition \ref{no2-term}.
Thus, we have by $(\ref{siki-psi2-e})$
\begin{align*}
\left(U(0,t)U_0(t,0)u_0,\psi\right)
&=\left( u_0,U_0(0,t)U(t,0)\psi\right)\\
&=\frac{1}{\left(\varphi_{0},\Phi\right)}\left( W_\Phi u_0,W_{\varphi_{0}}[U_0(0,t)U(t,0)\psi]\right)_{L^2(\mathbb{R}^{2n})}\\
&=\frac{1}{(\varphi_{0},\Phi)}\left(W_\Phi u_0,W_{\varphi_{0}}\psi+\int_0^tR_{\varphi_0}(s,x+s\xi,\xi;\psi)ds\right)_{L^2(\mathbb{R}^{2n})}.
\end{align*}

Using Lemma \ref{no2-term} and Schwarz's inequality, we obtain for $t'>t\geq T$
\begin{align*}
\left|\left( W_\Phi u_0,\int_{t}^{t'}R_{\varphi_0}(s,x+s\xi,\xi;\psi)ds\right)_{L^2({\mbox{supp}W_\Phi u_0})}\right|\\
&\hspace{-30mm}\leq\|W_\Phi u_0\|_{L^2(\mathbb{R}^{2n})}\cdot \int_t^{t'}\|R_{\varphi_0}(s,x+s\xi,\xi;\psi)\|_{L^2(\mathbb{R}^{2n}\setminus\Gamma_{a,R})} ds\\
&\hspace{-30mm}\leq C\|u_0\|\cdot \int_t^{t'}C\langle s\rangle^{-\delta'}\|\psi\| ds\\
&\hspace{-30mm}\leq C \langle t\rangle^{1-\delta}\|u_0\|\cdot\|\psi\|
\end{align*}
if $T=T(a,R)$ is sufficiently large.
This implies the existence of the wave operator $W_+(0)$.
\end{proof}
%
%
\section{Characterization of the wave operators}
In this section, we characterize the ranges of the wave operators by the wave packet transform.

\begin{Prop}\label{PropComp}
Suppose that (A) be satisfied.
Then we have
\[
\mathcal{R}(W_{\pm}(\tau))=D_{scat}^{\pm,\Phi}(\tau)
\]
for any $\Phi\in\mathcal{S}_{scat}$.
\end{Prop}

\begin{proof}
Similar to the previous section, we shall prove
\begin{align}
\label{siki-compl}
\mathcal{R}(W_+(0))=D_{scat}^{+,\Phi}(0).
\end{align}
In the other cases, they can be proved similarly.

Let $\Phi\in\mathcal{S}_{scat}$ and $\varepsilon>0$ fixed.
For simplicity, we write $W_+=W_+(0)$, $D_{scat}^+=D_{scat}^{+,\Phi}(0)$.

We first prove
\begin{align}
\label{siki-sub}
\mathcal{R}(W_+)\subset D_{scat}^+.
\end{align}
Let $f\in\mathcal{R}(W_+)$. Then since $W_{\Phi}^{-1}(C_0^\infty(\mathbb{R}^{2n}\setminus\{|\xi|=0\}))$ is dense in $L^2$,
there exists $g\in W_{\Phi}^{-1}(C_0^\infty(\mathbb{R}^{2n}\setminus\{|\xi|=0\}))$ such that
\begin{align}
\label{siki-d-kinzi}
\|f-W_+g\|<\varepsilon.
\end{align}

By the definition of $W_+$ and $(\ref{siki-psi})$, we have
\begin{align}
\label{siki-compp}
\lim_{t\rightarrow\infty}\|U(t,0)W_+g-e^{-itH_0}g\|= 0,\quad W_{\Phi(t)}[e^{-itH_0}g](x+t\xi,\xi)=W_{\Phi}g(x,\xi).
\end{align}

For $g\in W_{\Phi}^{-1}(C_0^\infty(\mathbb{R}^{2n}\setminus\{|\xi|=0\}))$,
there exist $a>0$ and $R>0$ such that
$
\mbox{supp}W_{\Phi}g\subset\mathbb{R}^{2n}\setminus\Gamma_{a,R}.
$

Thus from $(\ref{siki-d-kinzi})$, $(\ref{siki-compp})$, we obtain
\[
\lim_{t\rightarrow\infty}\|\chi_{\Gamma_{a,R}}W_{\Phi(t)}[U(t,0)f](x+t\xi,\xi)\|_{L^2(\mathbb{R}^{2n})}\leq\varepsilon.
\]

Thus $(\ref{siki-sub})$ is proved.

Next we prove
\begin{align}
\label{siki-sup}
\mathcal{R}(W_+)\supset D_{scat}^+.
\end{align}
If the limit
\begin{align}
\label{siki-inv-wo}
W_{+}^{-1}u_0=\mbox{s-}\hspace{-2mm}\lim_{t\rightarrow+\infty}e^{itH_0}U(t,0)u_0
\end{align}
exists for any $u_0\in D_{scat}^+$,
$(\ref{siki-sup})$ is obtained.

It suffices to prove $(\ref{siki-inv-wo})$ for $u_0\in\tilde{D}_{scat}^{+,\Phi}(0)$.
Indeed, from the definition of $D_{scat}^+$, for any $\varepsilon>0$ and $u_0\in D_{scat}^+$ there exists $\tilde{u_0}\in\tilde{D}_{scat}^{+,\Phi}(0)$
such that $\|u_0-\tilde{u_0}\|<\varepsilon$.

Let $u_0\in\tilde{D}_{scat}^{+,\Phi}(0)$.
Then we fix $a>0$ and $R>0$ satisfying
\begin{align}
\label{siki-gc}
\lim_{t\rightarrow\infty}\|\chi_{\Gamma_{a,R}}W_{\Phi(t)}[U(t,0)u_0](x+t\xi,\xi)\|_{L^2(\mathbb{R}^{2n})}=0
\end{align}
and write $\Gamma=\Gamma_{a,R}$, $\Gamma^c=\mathbb{R}^{2n}\setminus\Gamma$.

From Proposition \ref{wpt-prop}, we have
\begin{align*}
\left( e^{itH_0}U(t,0)u_0,\psi\right)
&=\left( U(t,0)u_0,e^{-itH_0}\psi\right)\\
&=\frac{1}{\left(\varphi(t),\Phi(t)\right)}\left( W_{\Phi(t)}[U(t,0)u_0],W_{\varphi(t)}[e^{-itH_0}\psi]\right)_{L^2(\mathbb{R}^{2n})}\\
&=\frac{1}{\left(\varphi_{0},\Phi\right)}\left( \chi_{\Gamma}(x-t\xi,\xi)W_{\Phi(t)}[U(t,0)u_0],W_{\varphi(t)}[e^{-itH_0}\psi]\right)_{L^2(\mathbb{R}^{2n})}\\
&+\frac{1}{\left(\varphi_{0},\Phi\right)}\left( \chi_{\Gamma^c}(x-t\xi,\xi)W_{\Phi(t)}[U(t,0)u_0],W_{\varphi(t)}[e^{-itH_0}\psi]\right)_{L^2(\mathbb{R}^{2n})},\\
\left( e^{it'H_0}U(t',0)u_0,\psi\right)
&=\left( U(t,0)u_0,U(t,t')e^{-it'H_0}\psi\right)\\
&=\frac{1}{\left(\varphi_{0},\Phi\right)}\left( \chi_{\Gamma}(x-t\xi,\xi)W_{\Phi(t)}[U(t,0)u_0],W_{\varphi(t)}[U(t,t')e^{-it'H_0}\psi]\right)_{L^2(\mathbb{R}^{2n})}\\
&+\frac{1}{\left(\varphi_{0},\Phi\right)}\left( \chi_{\Gamma^c}(x-t\xi,\xi)W_{\Phi(t)}[U(t,0)u_0],W_{\varphi(t)}[U(t,t')e^{-it'H_0}\psi]\right)_{L^2(\mathbb{R}^{2n})}.
\end{align*}

Thus the difference between the above equations is that
\begin{align*}
\Big(&e^{itH_0}U(t,0)u_0-e^{it'H_0}U(t',0)u_0,\psi\Big)\\
&=\frac{1}{\left(\varphi_{0},\Phi\right)}\left( \chi_{\Gamma}(x-t\xi,\xi)W_{\Phi(t)}[U(t,0)u_0],W_{\varphi(t)}[e^{-itH_0}\psi]-W_{\varphi(t)}[U(t,t')e^{-it'H_0}\psi]\right)_{L^2(\mathbb{R}^{2n})}\\
&+\frac{1}{\left(\varphi_{0},\Phi\right)}\left( \chi_{\Gamma^c}(x-t\xi,\xi)W_{\Phi(t)}[U(t,0)u_0],W_{\varphi(t)}[e^{-itH_0}\psi]-W_{\varphi(t)}[U(t,t')e^{-it'H_0}\psi]\right)_{L^2(\mathbb{R}^{2n})}.
\end{align*}

Using Schwarz's inequality and $(\ref{siki-gc})$, we obtain
\begin{align}
\label{siki-G}
\left|\left( \chi_{\Gamma}(x-t\xi,\xi)W_{\Phi(t)}[U(t,0)u_0],W_{\varphi(t)}[e^{-itH_0}\psi]-W_{\varphi(t)}[U(t,t')e^{-it'H_0}\psi]\right)_{L^2(\mathbb{R}^{2n})}\right|\nonumber\\
&\hspace{-90mm}\leq2\|\chi_{\Gamma}W_{\Phi(t)}[U(t,0)u_0](x+t\xi,\xi)\|\|\varphi_0\|\|\psi\|
\rightarrow 0.\quad(t\rightarrow+\infty)
\end{align}

Since $(\ref{siki-psi})$, $(\ref{siki-psi2})$ and the similar calculation of Proposition \ref{no2-term}, we have for $t'\geq t\geq T$
\begin{align}
\label{siki-Gc}
\Big|&\Big(\chi_{\Gamma^c}(x-t\xi,\xi)W_{\Phi(t)}[U(t,0)u_0],W_{\varphi(t)}[e^{-itH_0}\psi]-W_{\varphi(t)}[U(t,t')e^{-it'H_0}\psi]\Big)_{L^2(\mathbb{R}^{2n})}\Big|\nonumber\\
&=\Bigg|\Bigg( W_{\Phi(t)}[U(t,0)u_0],\chi_{\Gamma^c}(x-t\xi,\xi)\int_t^{t'} e^{-i\frac{1}{2}(t-s) |\xi|^2}
R_{\varphi_0}(s,x-(t-s)\xi,\xi;e^{-it'H_0}\psi)ds\Bigg)\Bigg|\nonumber\\
&=\Bigg|\Bigg( W_{\Phi(t)}[U(t,0)u_0](x+t\xi,\xi),\chi_{\Gamma^c}\int_t^{t'} e^{-i\frac{1}{2}(t-s) |\xi|^2}
R_{\varphi_0}(s,x+s\xi,\xi;e^{-it'H_0}\psi)ds\Bigg)\Bigg|\\
&\leq\|u_0\|\|\varphi_0\|\|\psi\|\langle t\rangle^{-\delta+1},\nonumber
\end{align}
if $T=T(a,R)$ is sufficiently large.
$(\ref{siki-inv-wo})$ follows from $(\ref{siki-G})$ and $(\ref{siki-Gc})$.
Combining $(\ref{siki-sub})$, $(\ref{siki-sup})$, we obtain $(\ref{siki-compl})$.
\end{proof}
Theorem \ref{main-theorem} is obtained by Proposition \ref{PropExis} and \ref{PropComp}

%
%
\section{Remarks}
In this section, we consider the remarks in Section 1.
%
%

The assertion of Remark \ref{rem-asu} follows the estimate
\[
\|I_1(s,x,\xi;\psi)\|\leq Ch(s)\|\psi\|,
\]
where $I_1(s,x,\xi;\psi)$ is the term of Proposition \ref{no2-term} and $C$ is some constant.

We prove Remark \ref{rem-eigen} as follows.
\begin{proof}[Proof of Remark \ref{rem-eigen}]
Since the existence of the wave operators and the fact $\mathcal{R}(W_{\pm}(\tau))=D_{scat}^{\pm}(\tau)$ are proved in Theorem 1,
it suffices to show $D_{scat}^{\pm}(\tau)\perp\mathcal{H}_{p}$.

Similar to the previous sections, we shall prove for $\tau=0$ and $t\rightarrow+\infty$.
Let $u_0\in \tilde{D}_{scat}^{+}()$ and $\Phi\in\mathcal{S}_{scat}$.
Then there exist $a,R>0$  such that
\begin{align}
\label{siki-aa}
\lim_{t\rightarrow\infty}\|\chi_{\Gamma_{a,R}}W_{\Phi(t)}[e^{-itH}u_0](x+t\xi,\xi)\|_{L^2(\mathbb{R}^{2n})}=0.
\end{align}

On the other hand, if we take
$\omega\in\mathcal{H}_{p}(H)$,
it suffices that
\begin{align}
\label{siki-ab}
e^{-itH}\omega=e^{-it\lambda}\omega.
\end{align}

Then taking $\varphi_0\in \mathcal{S}$ similarly to section 3, we get
\begin{align*}
(u_0,\omega)
&=(e^{-itH}u_0,e^{-itH}\omega)\\
&=\frac{1}{\left(\varphi(t),\Phi(t)\right)}\left(W_{\Phi(t)} [e^{-itH}u_0],W_{\varphi(t)}[e^{-itH}\omega]\right)_{L^2(\mathbb{R}^{2n})}\\
&=\frac{1}{\left(\varphi_{0},\Phi\right)}\left(\chi_{\Gamma_{a,R}}W_{\Phi(t)} [e^{-itH}u_0](x+t\xi,\xi),W_{\varphi(t)}[e^{-itH}\omega](x+t\xi,\xi)\right)_{L^2(\mathbb{R}^{2n})}\\
&+\frac{1}{\left(\varphi_{0},\Phi\right)}\left(W_{\Phi(t)} [e^{-itH}u_0](x+t\xi,\xi),
\chi_{\{\mathbb{R}^{2n}\setminus\Gamma_{a,R}\}}W_{\varphi(t)}[e^{-itH}\omega](x+t\xi,\xi)\right)_{L^2(\mathbb{R}^{2n})}.
\end{align*}
By $(\ref{siki-aa})$, the first term is estimated by
\begin{align*}
\Big|\frac{1}{\left(\varphi_{0},\Phi\right)}&\left(\chi_{\Gamma_{a,R}}W_\Phi [e^{-itH}u_0](x+t\xi,\xi),W_{\varphi_{0}}[e^{-itH}\omega](x+t\xi,\xi)\right)_{L^2(\mathbb{R}^{2n})}\Big|\\
&\leq C\|\chi_{\Gamma_{a,R}}W_\Phi [e^{-itH}u_0](x+t\xi,\xi)\|_{L^2(\mathbb{R}^{2n})}\|\omega\|\rightarrow0.\quad(t\rightarrow\infty)
\end{align*}
By $(\ref{siki-ab})$, the second term is estimated by
\begin{align}
\label{siki-ac}
\Big|\frac{1}{\left(\varphi_{0},\Phi\right)}&\left(W_{\Phi(t)} [e^{-itH}u_0](x+t\xi,\xi),
\chi_{\{\mathbb{R}^{2n}\setminus\Gamma_{a,R}\}}W_{\varphi(t)}[e^{-itH}\omega](x+t\xi,\xi)\right)_{L^2(\mathbb{R}^{2n})}\Big|\nonumber\\
=\Big|\frac{1}{\left(\varphi_{0},\Phi\right)}&\left(W_{\Phi(t)} [e^{-itH}u_0](x+t\xi,\xi),
\chi_{\{\mathbb{R}^{2n}\setminus\Gamma_{a,R}\}}W_{\varphi(t)}[e^{-it\lambda}\omega](x+t\xi,\xi)\right)_{L^2(\mathbb{R}^{2n})}\Big|\nonumber\\
&\leq C\|u_0\|\|\chi_{\{\mathbb{R}^{2n}\setminus\Gamma_{a,R}\}}W_{\varphi(t)}[\omega](x+t\xi,\xi)\|_{L^2(\mathbb{R}^{2n})}\rightarrow0.\quad(t\rightarrow\infty)
\end{align}
$(\ref{siki-ac})$ follows the density argument and Lemma \ref{kuroda-lem}.
Therefore we obtain $D_{scat}^{\pm}(0)\perp\mathcal{H}_{p}(H)$.
\end{proof}

Since $W_{\Phi}^{-1}(C_0^\infty(\mathbb{R}^{2n}\setminus\tilde{\Gamma}_{0,1}))$ is dense in $L^2$, 
the same proof in the previous section follows $(\ref{AeqD})$.
Thus we get the claim of Remark \ref{rem-ana}.


\end{document}